\documentclass[12pt,a4paper]{amsart}

\usepackage{geometry}
\usepackage[T1]{fontenc}
\usepackage[utf8]{inputenc}
\usepackage{authblk}
\usepackage{amsmath}
\usepackage{amssymb}
\usepackage{amsfonts}
\usepackage{amsthm}
\usepackage{enumerate}
\usepackage{amsfonts}
\usepackage{tensor} 
\usepackage{lmodern}
\usepackage{upgreek}
\usepackage{mathrsfs}
\usepackage{nicefrac}
\usepackage{verbatim}
\usepackage{amsmath}
\usepackage[all, cmtip]{xy}\usepackage{amssymb}
\usepackage{amsthm}
\usepackage{extpfeil}
\usepackage{bbm}
\usepackage[colorlinks,linkcolor=blue,citecolor=gray,pagebackref=false,urlcolor=blue,pdftex]{hyperref}

 \usepackage[foot]{amsaddr}

\usepackage{lipsum}% http://ctan.org/pkg/lipsum
\makeatletter

\usepackage{graphicx}

\usepackage{hyperref}
\usepackage[all]{xy}
\usepackage{tikz}
\usepackage{relsize}
\usepackage{comment}
\usepackage{mathtools}

\usepackage{pgfplots}
\usepackage{subfig}
\usepackage{float}

\theoremstyle{plain}
\newtheorem{thm}{Theorem}
\newtheorem{lem}{Lemma}

\newtheorem*{prb*}{Problem}
\newtheorem{prop}[thm]{Proposition}
\newtheorem{rmk}{Remark}
\newtheorem{ex}{Example}
\newtheorem{cor}[thm]{Corollary}

\theoremstyle{definition}

\theoremstyle{remark}

\newcommand{\cm}[1]{}

\newcommand\wt[1]{\widetilde{#1}}
\newcommand{\R}{\mathbb{R}}
\newcommand{\Z}{\mathbb{Z}}

\begin{document}

\title{On Codimension one Embedding of Simplicial Complexes}
\author{Anders Bj\"{o}rner}

\address{Royal Institute of Technology, Department of Mathematics, S-100 44, Stockholm, Sweden.}
\email{bjorner@math.kth.se} 
\author{Afshin Goodarzi}

\address{Free University of Berlin, Department of Mathematics, Discrete Geometry Group, 14195, Berlin,
Germany.}

\email{goodarzi@math.fu-berlin.de}

\dedicatory{Dedicated to the memory of Ji{\v{r}}{\'{\i}} Matou{\v{s}}ek}

%\renewcommand\Authands{ and }
% ------------------------------------------------------------------------

\maketitle

% ----------------------------------------------- SECTION------------------------------------------------------------------

\begin{abstract}

We study $d$-dimensional simplicial complexes that are PL
embeddable in $\mathbb{R}^{d+1}$. It is shown that such a complex
must satisfy a certain homological condition. 
%In the case of graphs (the $d=1$ case) this condition was introduced by 
%S.~Mac Lane~\cite{MacLane37}.
The existence of  this obstruction allows us to provide a systematic approach
to deriving upper bounds for the number of 
top-dimensional faces of such complexes,
particularly in low dimensions.

\end{abstract}

% ----------------------------------------------- SECTION------------------------------------------------------------------

\section*{Introduction}

%In this section we introduce our problem and mention some results,

The question of embeddability of a $d$-dimensional 
simplicial complex into $k$-dimensional Euclidean space $\mathbb{R}^{k}$
has a long history. 
%$\mathbb{R}^{k}$, for various values of $d$ and $k$ and various
%notions of embedding,
In the following section we sketch some of this
background. Technical definitions and details appear in later sections. 
See J. Matou{\v{s}}ek's
book  \cite[chapter 5]{Mato} and his paper with M. Tancer
and U. Wagner \cite{MTW11}  for nice introductions to the field.

In this note
%generalising one direction of Mac Lane's result, 
we provide a homological obstruction to codimension one ($k=d+1$) 
piecewise linear (PL) embeddability of simplicial complexes. For
the case of graphs ($d=1$) this kind of obstruction was used by
S.~Mac Lane~\cite{MacLane37}
in his work on planarity.

As corollaries we derive upper bounds for the number of top-dimensional faces in a complex with codimension one PL embedding,
 in terms of the lower dimensional face numbers and Betti numbers. 
% with respect to the number of faces of lower dimensions.
For instance, we show that
\[ f_d(\Sigma)\leq 
\frac{\mathrm{g}(\Sigma)}
{\mathrm{g} (\Sigma)-2}
%\frac{d+2}{d} 
\left(\left( \sum_{i=1}^d(-1)^{i-1}\left(f_{d-i}(\Sigma)-\beta_{d-i}(\Sigma)\right)
\right)-1\right),\]
where $f_i(\Sigma)$ is the number of faces of dimension $i$, $\beta_{i}$ is the Betti number in dimension $i$, and $\mathrm{g}(\Sigma)$ is the girth (smallest size of
a $d$-cycle in non-zero homology).
See Theorem \ref{thm1b} for details. For $d=1$  
and $\mathrm{g}(\Sigma)=3$ this specializes to Euler's $3n-6$ upper bound for the maximal
number of edges of a planar graph. 
%mentioned in the following section.

 The method used enables us to provide a unified approach and
to give more detailed versions of face number inequalities for such complexes in low dimensions. 
For instance, we obtain that 
\[ f_2(\Sigma)\leq 2\left(f_1(\Sigma)-f_0(\Sigma)-\beta_1(\Sigma)\right)\]
for any connected $2$-dimensional complex $\Sigma$ that PL embeds into 
$\R^3$, see Proposition \ref{2to3}.
Furthermore, we give a new upper bound for the number of facets of 
 complexes with codimension one PL embedding, in terms
 only of the number of vertices. 
%Our upper bound 
This slightly improves the upper bound given by Dey and Pach~\cite{DeyPach}. 

Finally, some of our face number  inequalities 
are adapted to the case of balanced complexes,
i.e., complexes whose
$1$-skeleton is $(d+1)$-colorable in the graph-theoretic sense.

\section*{Background}

The concept of planarity has been of interest to mathematicians ever since the subject of graph theory was founded. For instance, the impossibility 
for a planar graph on $n$ ($\geq 3$) vertices
of having more than $3n-6$ edges was mentioned in a letter from L.~Euler to C.~Goldbach in 1750, see~\cite[p. 75]{BLW}. 

A topological characterisation of planarity was given by K.~Kuratowski in 1929 and independently (a few months later) by O.~Frink and P.A.~Smith. This result asserts that a finite graph is planar if and only if it does not contain a subgraph homeomorphic to $K_5$ or $K_{3,3}$. Since then other characterisations of planarity have been given. Among them one can mention the more combinatorial approaches by H.~Whitney~\cite{W32} and S.~Mac Lane~\cite{MacLane37}, and the more topological approach of H.~Hanani and W.T.~Tutte (see~\cite{Tutte}, for instance). 

What can be said about the situation in higher dimensions?
Let $\Sigma$ be a finite $d$-dimensional simplicial complex. It was known since the early days of topology that $\Sigma$ is linearly embeddable into $\mathbb{R}^{2d+1}$. In his 1933 article, E.~R.~van Kampen~\cite{vanK} showed that this result is best possible, by presenting $d$-dimensional complexes (now known as the van Kampen--Flores complexes) that do not embed into $\mathbb{R}^{2d}$. Thus, the natural question is, for $d\leq k\leq 2d$, does $\Sigma$ admit an embedding into $\R^k$? 
The  most intensively investigated cases
are when $k=2d$ or $k=d+1$. 
Note that these are the two natural generalisations 
to higher dimensions of the concept of planarity. 

There is no satisfactory analogue of Kuratowski's characterisation in higher dimensions. Indeed, for every $d>1$ and $d+1\leq k\leq 2d$, J.~Zaks~\cite{Zaks2} constructed infinitely many pairwise non-homeomorphic 
$d$-dimensional complexes that are mini\-mal with
respect to the property of being not embeddable in $\R^{k}$.

Based on the aforementioned work of van Kampen, in 1957 A.~Shapiro~\cite{Shapiro57} introduced the \emph{van Kampen obstruction}; a cohomological obstruction to embeddability of $d$-dimensional complexes into $\mathbb{R}^{2d}$. See~\cite{MTW11} for a geometric description. The van Kampen obstruction can be seen as a higher-dimensional analogue of the Hanani--Tutte theorem, though the strong version of Hanani--Tutte theorem appeared much later in~\cite{Tutte}. 

%In this note, generalising one direction of Mac Lane's result, we provide a homological obstruction to codimension one, i.e. the $k=d+1$ case, piecewise linear (PL) embeddability of simplicial complexes. As a corollary we get upper bounds for the number of facets in a complex that admits a codimension one PL embedding with respect to the number of faces of lower dimensions. This enables us to provide a unified approach and more detailed versions of face numbers of such complexes in low dimensions. Furthermore, we provide a new upper bound for the number of facets of such complexes in terms of the number of vertices. Our upper bound slightly improves the upper bound given by Dey and Pach~\cite{DeyPach}. 

\section{Embedding}

 A simplicial complex $\Sigma$ is said to admit a \emph{linear embedding} into $\R^k$ if $\Sigma$ has a geometric realisation $\|\Sigma\|$ in $\R^k$. %Menger Embedding Theorem (1926) asserts that every  $d$-dimensional simplicial complex admits a linear embedding into $\R^{2d+1}$. See, for example, the book~\cite{Mato} by J.~Matou{\v{s}}ek for details. 
 More generally, $\Sigma$ admits a \emph{topological embedding} into $\R^k$ if there is a continuous injection $\|\Sigma\|\hookrightarrow \R^k$, from some geometric realisation of $\Sigma$ to $\R^k$. 
 An intermediate concept
 %between linear and topological embedding 
 is that of \emph{PL embedding}.
 We say that $\Sigma$ is \emph{piecewise linear} (\emph{PL}) embeddable into $\R^k$ if there is a subdivision of 
$\|\Sigma\|$ that linearly embeds into $\R^k$. In this paper we focus on PL embeddings. 
 
 It is a consequence of Steinitz' Theorem~\cite[Lect.~4]{Z} that every planar graph can be drawn in the plane with straight edges. However, for higher dimensional objects the situation is more complicated.

\begin{ex}[Brehm's triangulated M\"{o}bius strip] 
{\rm  In~\cite{NonLinearPL}, Brehm presented a triangulation of the M\"{o}bius strip that can not be geometrically realised in $\R^3$. The idea is simple but elegant: A non null-homotopic curve, different from the center line, and the boundary curve of the M\"{o}bius strip are linked together, with absolute value of the linking number at least $2$. This can easily be visualised by, for instance, considering the blue curve on the left hand side of Figure 1 below. Now, triangulate the M\"{o}bius strip in such a way that the blue curve and the boundary curve are induced triangles; see the right hand side of Figure 1. Two triangles with straight edges in $\R^3$ are either the unlink or the Hopf link. Hence, these two triangles cannot be realised by straight edges. Iterated simplicial suspensions produces examples of $d$-dimensional complexes that are PL embeddable into $\R^{d+1}$ but do not admit a linear embedding. 
}
 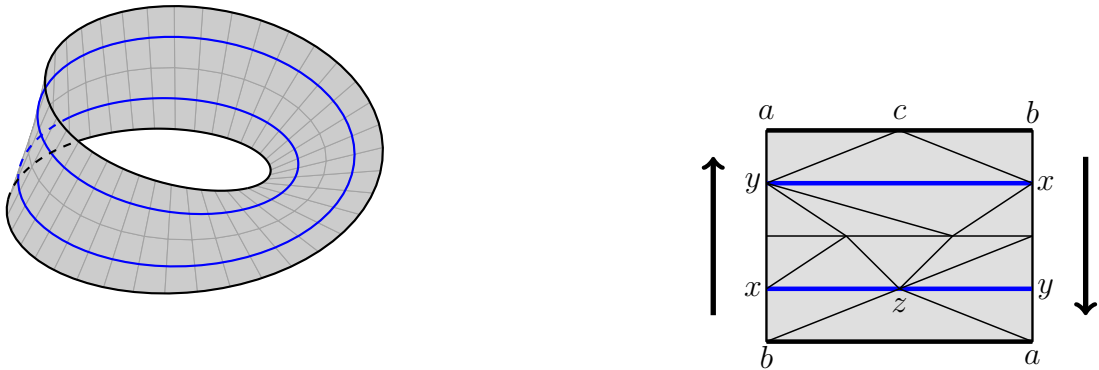
\begin{figure}[H]\label{Fig1}
 
 \centering
 
 \subfloat{
\begin{tikzpicture}
  \begin{axis}[
    hide axis,
    view = {40}{60}
  ]
  \addplot3 [
    surf,
colormap={blackwhite}{gray(0cm)=(0.8); gray(1cm)=(0.8)},
   %colormap/blackwhite,
 %  shader     = faceted interp,
    point meta = x,
    samples    = 40,
    samples y  = 5,
   % z buffer   = sort,
    domain     = 0:360,
    y domain   =-0.8:0.8
  ] (
    {(1+0.5*y*cos(x/2)))*cos(x)},
    {(1+0.5*y*cos(x/2)))*sin(x)},
    {0.5*y*sin(x/2)}
  );

    \addplot3 [
samples=100,
    domain=-360:-185, % The domain needs to be adjusted manually,
                     % depending on the camera angle, unfortunately
    samples y=0,
  style=  {thick,blue}
  ] (
    {(1+0.2*cos(x/2)))*cos(x)},
    {(1+0.2*cos(x/2)))*sin(x)},
    {0.2*sin(x/2)}
  );

    \addplot3 [
samples=100,
    domain=-185:-140, % The domain needs to be adjusted manually,
                     % depending on the camera angle, unfortunately
    samples y=0,
  style=  {dashed, thick,blue}
  ] (
    {(1+0.2*cos(x/2)))*cos(x)},
    {(1+0.2*cos(x/2)))*sin(x)},
    {0.2*sin(x/2)}
  );

    \addplot3 [
samples=100,
    domain=-140:360, % The domain needs to be adjusted manually,
                     % depending on the camera angle, unfortunately
    samples y=0,
  style=  {thick,blue}
  ] (
    {(1+0.2*cos(x/2)))*cos(x)},
    {(1+0.2*cos(x/2)))*sin(x)},
    {0.2*sin(x/2)}
  );

    \addplot3 [
samples=100,
    domain=-140:360, % The domain needs to be adjusted manually,
                     % depending on the camera angle, unfortunately
    samples y=0,
   thick
     ] (
    {(1+0.4*cos(x/2)))*cos(x)},
    {(1+0.4*cos(x/2)))*sin(x)},
    {0.4*sin(x/2)}
  );
  
    \addplot3 [
samples=100,
    domain=-192:-140, % The domain needs to be adjusted manually,
                     % depending on the camera angle, unfortunately
    samples y=0,
  style={dashed,  thick}
     ] (
    {(1+0.4*cos(x/2)))*cos(x)},
    {(1+0.4*cos(x/2)))*sin(x)},
    {0.4*sin(x/2)}
  );

    \addplot3 [
samples=100,
    domain=-360:-192, % The domain needs to be adjusted manually,
                     % depending on the camera angle, unfortunately
    samples y=0,
    thick
     ] (
    {(1+0.4*cos(x/2)))*cos(x)},
    {(1+0.4*cos(x/2)))*sin(x)},
    {0.4*sin(x/2)}
  );

  \end{axis}
  
\end{tikzpicture}

\hspace{3cm}

\begin{tikzpicture}[line join=bevel,z=-5.5,scale=0.35]

\coordinate (a) at (0,8);
\coordinate (b) at (0,0);
\coordinate (c) at (0,6);
\coordinate (d) at (0,4);
\coordinate (e) at (0,2);
\coordinate (f) at (5,8);
\coordinate (g) at (5,2);
\coordinate (h) at (3,4);
\coordinate (i) at (7,4);
\coordinate (j) at (10,8);
\coordinate (k) at (10,6);
\coordinate (l) at (10,4);
\coordinate (m) at (10,2);
\coordinate (n) at (10,0);

\draw[thick,fill=gray!25] (a)--(b)--(n)--(j)--cycle;

\draw [line width=0.6mm,black] (a) -- (j);
\draw [line width=0.6mm,black] (b) -- (n);

\draw [line width=0.6mm,blue] (c) -- (k);

\draw [line width=0.6mm,blue] (m) -- (e);

\draw [line width=0.2mm,black] (d) -- (l);
\draw [line width=0.2mm,black] (a) -- (b);
\draw [line width=0.2mm,black] (j) -- (n);
\draw [line width=0.2mm,black] (f) -- (c);
\draw [line width=0.2mm,black] (f) -- (k);
\draw [line width=0.2mm,black] (g) -- (h);
\draw [line width=0.2mm,black] (g) -- (i);
\draw [line width=0.2mm,black] (c) -- (i);
\draw [line width=0.2mm,black] (g) -- (b);
\draw [line width=0.2mm,black] (g) -- (n);
\draw [line width=0.2mm,black] (k) -- (i);
\draw [line width=0.2mm,black] (g) -- (l);
\draw [line width=0.2mm,black] (c) -- (h);
\draw [line width=0.2mm,black] (e) -- (h);

\draw[line width=2pt,->] (-2,1) -- (-2,7);

\draw[line width=2pt,->] (12,7) -- (12,1);

\node at (0,8.7){$a$};
\node at (5,8.7){$c$};
\node at (10,8.7){$b$};

\node at (10,-0.6){$a$};
\node at (0,-0.6){$b$};
\node at (10.5,6){$x$};
\node at (-0.5,2){$x$};

\node at (10.5,2){$y$};
\node at (-0.5,6){$y$};
\node at (5,1.4){$z$};

\end{tikzpicture}

}
\caption{Brehm's triangulated M\"{o}bius strip }
\end{figure}
\end{ex}
The difference between linear and PL embedding is even more dramatic. One can show that the problem of linear embeddability is algorithmically decidable. On the other hand, it is shown in~\cite[Theorem 1.1]{MTW11} that codimension one PL embeddability is algorithmically undecidable for $d\geq 4$. See~\cite{MTW11} for a thorough discussion.

Let us also remark that topological and PL embeddings do not coincide in codimension one. In fact, by the double suspension theorem~\cite{cannon}, the suspension of the Poincar\'{e} homology $3$-sphere topologically embeds into $\R^5$. However, it does not admit a PL embedding into $\R^5$~\cite[p.~576]{Uli}.

\section{Main Results}
 
Let $\Sigma$ be a $d$-dimensional simplicial complex. We consider 
simplicial homology of $\Sigma$ with $\Z_2$ coefficients. Let \(c=\sum\epsilon_{\sigma}\sigma\) be a $d$-chain, where the sum is over all $d$-dimensional faces of $\Sigma$ and $\epsilon_\sigma\in\Z_2$. We let the \emph{support} $\mathrm{supp}(c)$ of $c$ be the set of all $d$-faces $\sigma$ such that $\epsilon_\sigma=1$. 

 Let us say that a basis $\mathcal{B}$ of $H_d(\Sigma;\mathbb{Z}_2)$ is \emph{$m$-complete} if every $d$-dimensional face of $\Sigma$ appears in the support of at most $m$ elements in $\mathcal{B}$. When $d=1$, this definition agrees with Mac Lane's concept of \emph{$m$-fold complete set of cycles} for graphs. He showed that having a $2$-fold complete set of cycles is equivalent to planarity for graphs~\cite{MacLane37}. In this section we generalise one direction of Mac Lane's result. Before doing so, we need to show the following topological invariance property. 
 
 \begin{lem}\label{singular}
 Let $\Sigma$ and $\Gamma$ be two triangulations of a $d$-dimensional topological space $X$. Then $\Sigma$ has an $m$-complete basis if and only if $\Gamma$ has a $m$-complete basis. 
 \end{lem}
 \begin{proof}
 Let $H_d(X;\mathbb{Z}_2)$ be the singular homology group of $X$ (this is the only place in this paper where we use singular homology theory). We refer to the book~\cite{Munk} by Munkres for the definition and properties of the singular homology.
 
  Let $H_d(X;\mathbb{Z}_2)=\mathbb{Z}_2^r$. We can always assume that there are $d$-dimensional pseudomanifolds $M_1,\ldots,M_r$ and continuous maps $f^i:M_i\rightarrow X$, for $1\leq i\leq r$, so that the $d$-dimensional homology classes of $X$ are $f^i_\sharp[M_i]$, where $[M_i]$ is the fundamental class of $M_i$. We claim that a triangulation of $X$ has an $m$-complete basis if and only if there is a choice of $M_i$ and $f^i$ such that for any subset $I$ of $\{1,2,\ldots,r\}$ of size greater than $m$ one has \[\dim\left(\bigcap_{i\in I} f^i(M_i)\right)<d.\]
  Observe that once the claim is verified the desired statement is immediate. However, the verification of the claim is standard and we leave it to the reader. 
 \end{proof}
 
\begin{rmk}
{\rm Since we are working with $\mathbb{Z}_2$ coefficients, it follows from a result by Thom that, $M_1,\ldots,M_r$ in the proof of Lemma~\ref{singular} can be taken to be closed manifolds. See, for instance, \cite[p. 343]{steenrod}. 
}
\end{rmk}

\begin{thm}\label{2fold}
Let $\Sigma$ be a $d$-dimensional simplicial complex that admits a PL embedding into $\R^{d+1}$. Then $H_d(\Sigma;\mathbb{Z}_2)$ has a $2$-complete basis. 
\end{thm}

\begin{proof} First notice that, by Lemma~\ref{singular}, $\Sigma$ has a $2$-complete basis if and only if any subdivision of $\Sigma$ has this property. This allows us to replace $\Sigma$ by a subdivision of $\Sigma$ if needed. Also, observe that since $\Sigma$ is PL embeddable into $\mathbb{R}^{d+1}$, then $\Sigma$ is PL embeddable into the $(d+1)$-simplex $\Delta_{d+1}$. Thus there is a subdivision $\Sigma'$ of $\Sigma$ and a subdivision $\mathbf{B}$ of $\Delta_{d+1}$ such that $\Sigma'$ is a subcomplex of $\mathbf{B}$. So, we may assume that  $\Sigma'$ is a subcomplex of a simplicial $(d+1)$-sphere $\mathbf{S}$, say by embedding $\mathbf{B}$ into a hyperplane $H$ of $\R^{d+2}$ and taking $\mathbf{S}=\{p\}\ast\partial\mathbf{B}\cup\mathbf{B}$, where $p$ is a point outside $H$ and $\ast$ denotes the simplicial cone.

Now, set $r:=\beta_d(\Sigma';\Z_2)+1$. There is nothing to prove if $r=1$. So, we may assume that $r>1$. It follows from Alexander duality~\cite[Theorem 71.1]{Munk} that $\|\mathbf{S}\|-\|\Sigma'\|$ has $r$ connected components, say $K_1,\ldots,K_r$. For $1\leq j\leq r$, let $c_j$ be the formal sum (modulo $2$) of all facets $F$ of $\mathbf{S}$ such that the barycenter of $F$ lies in $K_j$. Let $b_j$ be the boundary $\partial_{d+1} c_j$ of $c_j$. Notice that $b_j\neq 0$, since $r>1$ and therefore, $c_j$ cannot be a $(d+1)$-cycle. 

We will show that $b_1,\ldots,b_{r-1}$ form a $2$-complete basis for $H_d(\Sigma';\mathbb{Z}_2)$. 

Let $\sigma\in\mathrm{supp}(b_j)$ for some $1\leq j\leq r$. Then $\sigma$ is a facet of $\Sigma'$. Otherwise, the facets $F_\sigma$ and $F'_\sigma$ of $\mathbf{S}$ that contain $\sigma$ lie in the same connected component $K_j$. This implies that $F_\sigma$ and $F'_\sigma$ are in the support of $c_j$. Hence, $\sigma\notin\mathrm{supp}(b_j)$, which is a contradiction. Also, observe that there exists exactly one $i\neq j$ such that $\sigma\in\mathrm{supp}(b_i)$, since every codimension one face of $\mathbf{S}$ is in exactly two facets.

It is immediate that $\partial_d b_j=\partial_d\partial_{d+1} c_j=0$, hence every $b_j$ is a $d$-cycle in $\mathbf{S}$. However, since $\mathrm{supp}(b_j)$ is a subset of the set of faces of $\Sigma'$, then every $b_j$ is a $d$-cycle in $\Sigma'$. 

Finally, we have that $\sum_{i\in A} b_i\neq 0$ for all proper subsets $A$ of $\{1,\ldots,r\}$. Otherwise, 
\[\partial_{d+1}\left(\sum_{i\in A} c_i\right)= \left(\sum_{i\in A} \partial_{d+1}c_i\right)=\sum_{i\in A} b_i=0,\]
that is, the subcomplex of $\mathbf{S}$ whose set of facets are $c_i$, $i\in A$, has non-trivial $(d+1)$-dimensional homology. However, this cannot happen, since every proper subcomplex of $\mathbf{S}$ has trivial $(d+1)$-dimensional homology. Therefore, $b_1,\ldots,b_{r-1}$ is a $2$-complete basis for $H_d(\Sigma';\mathbb{Z}_2)$, as promised.  
\end{proof}
\begin{rmk}{\rm 
It might be possible that the conclusion of Theorem~\ref{2fold} is still valid if we consider the more general case of topological embedding. However, since we use Alexander duality, our method would not be directly applicable in that general setting. 
}
\end{rmk}
Notice that the converse to Theorem~\ref{2fold} is obviously false for all $d>1$. For instance, there are $d$-manifolds that do not admit an embedding into $\R^{d+1}$; non-orientable manifolds for example. In fact, it follows from Alexander duality that if $\Sigma$ is embeddable into the $(d+1)$-sphere $\mathbb{S}^{d+1}$, then the cohomology  $H^d(\Sigma;\Z)$ is isomorphic to $\wt{H}_0(\mathbb{S}^{d+1}\setminus\Sigma;\Z)$ and, thus, is torsion-free.

Having Theorem~\ref{2fold} in mind it is tempting to conjecture that if a $d$-dimensional simplicial complex $\Sigma$ embeds into $\R^{d+m-1}$, then $H_d(\Sigma;\mathbb{Z}_2)$ has an $m$-complete basis. The 
following example shows that this is not the case. 

\begin{ex}{\rm 
Let $n$ be an integer and let $\Delta$ be the $2$-dimensional complex obtained by suspending the complete bipartite graph $K_{n,n}$. Clearly, $f(\Delta)=(n+2,n^2+2n,2n^2)$ and $\beta_2(\Delta)=n^2-2n+1$. On the other hand, $\Delta$ (being a suspension of a complex embeddable in $3$-space) is embeddable into $\R^4$. However, we show that for large enough $n$, $H_2(\Delta;\mathbb{Z}_2)$ does not have a $3$-complete basis. First observe that if $\Omega$ is a minimal cycle in $\Delta$, then $\Omega$ has at least $8$ triangles. Now, let $\mathcal{B}$ be a basis for $H_2(\Delta;\mathbb{Z}_2)$ and let $M$ be the $n^2-2n+1$ by $2n^2$ $\{0,1\}$-matrix whose rows are labeled by the elements $\Omega$ of $\mathcal{B}$ and whose columns are labeled by the facets of $\Delta$,
and for which the entry $(F,\Omega)$ is the coefficient of $F$ in $\Omega$. Since the number of facets with non-zero coefficient in each element of $\mathcal{B}$ is at least $8$, the minimum number of $1$s in $M$ is $8(n^2-2n+1)$. On the other hand, if $H_2(\Delta;\mathbb{Z}_2)$ has a $3$-complete basis, then the maximum number of $1$s in $M$ must be $3$ times the number of facets, that is, $6n^2$. Therefore, if $n$ is large enough then $H_2(\Delta;\mathbb{Z}_2)$ does not have a $3$-complete basis. 
}\end{ex}

\section{Face Numbers}

In this section we provide upper bounds for the number of top dimensional faces of complexes that admit a codimension one embedding in terms of the lower dimensional face numbers and Betti numbers. 

For a $d$-dimensional simplicial complex $\Sigma$, with non-trivial top Betti number, let us define the \emph{girth} of $\Sigma$,
denoted $\mathrm{g}(\Sigma)$,
 to be the minimum number of $d$-dimensional faces of a subcomplex with non-zero $d$-dimensional Betti number.
%We denote the girth of $\Sigma$ by $\mathrm{girth}(\Sigma)$. 
This notion extends the graph theoretic notion of girth
as the minimal size of a circuit. 
 If $\beta_{d}(\Sigma)=0$ we define 
 the girth to be $d+2$. Note that 
the girth of a $d$-dimensional complex satisfies
$\mathrm{g}(\Sigma) \ge d+2$.

\begin{thm}\label{thm1a}
Let $\Sigma$ be a $d$-dimensional simplicial complex such that
$H_d(\Sigma;\mathbb{Z}_2)$ admits a $2$-complete basis.
%that admits a PL embedding into $\R^{d+1}$. Then,
%\begin{enumerate}[{\rm (a)}]
Then
\begin{equation}\label{eq2}
\mathrm{g}(\Sigma)\left(\beta_d(\Sigma;\Z_2)+1\right)\leq 2f_d(\Sigma).
\end{equation}
%\item\label{item2} $f_d(\Sigma)\leq (1+2/d)(f_{d-1}(\Sigma)-1).\]
%\end{enumerate}
\end{thm}

\begin{proof} Let $r$ and $b_1,\ldots,b_r$ be as defined in the proof of Theorem~\ref{2fold}. On the one hand, for $1\leq j\leq r$, $\mathrm{supp}(b_j)$ has at least $\mathrm{g}(\Sigma)$ elements. On the other hand, a $d$-dimensional face of $\Sigma$ appears, if at all, in the support of two of the $b_j$'s. Therefore, $\mathrm{g}(\Sigma)r\leq 2f_d(\Sigma)$, as desired. 
\end{proof}

To help simplify the notation, let
$\delta_{j} = f_{j}(\Sigma)-\beta_{j}(\Sigma;\Z_2)$,
for all $j$. Then, let
$$ \chi_{j-1}(\Sigma) =
\sum_{i=1}^j(-1)^{i-1} \delta_{j-i}
%\left(f_{j-i}(\Sigma)-\beta_{j-i}(\Sigma;\Z_2)\right).
$$

%$$ \chi_{j-1}(\Sigma) =
%\sum_{i=1}^j(-1)^{i-1}\left(f_{j-i}(\Sigma)-\beta_{j-i}(\Sigma;\Z_2)\right).

It follows from the rank-nullity theorem that $\chi_{j-1}(\Sigma) \ge0$ for all $j$. 
These inequalities, sometimes called the 
{\em strong Morse inequalities}, are discussed in Milnor \cite{Mil63},
and appear in slightly sharper form in \cite{BK}.

%Note that 
%$$ \chi_{d-1}(\Sigma) = f_{d}(\Sigma)-\beta_{d}(\Sigma;\Z_2)
%$$
%by the Euler-Poincar\'{e} formula.

\begin{thm}\label{thm1b}
Let $\Sigma$ be a $d$-dimensional simplicial complex that admits a PL embedding into $\R^{d+1}$. Then,
%f_d(\Sigma)\leq \frac{\mathrm{g}(\Sigma)}
%{\mathrm{g} (\Sigma)-2}
%(\chi_{d-1}(\Sigma)-1).
\begin{equation}\label{eq3}
f_d(\Sigma)\leq \frac{\mathrm{g}(\Sigma)}
{\mathrm{g} (\Sigma)-2}
(\delta_{d-1} - \delta_{d-2} + \delta_{d-3} - \cdots + \delta_{d-k} -1)
\end{equation}

for all odd $k\ge 1$.
\end{thm}

\begin{proof}
Our point of departure is the inequality (1) of Theorem~\ref{thm1a}. Replace $\beta_d(\Sigma;\Z_2)$ in the left hand side of the inequality by the 
right hand side of the following form of the Euler-Poincar\'{e} formula:
$$ \beta_{d}(\Sigma;\Z_2) = f_{d}(\Sigma)-\chi_{d-1}(\Sigma),
$$
and then simplify and use $\chi_{d-k-1}\ge 0$
to get the desired inequality. 
\end{proof}

\begin{cor}\label{cor1c}
Let $\Sigma$ be a $d$-dimensional simplicial complex that admits a PL embedding into $\R^{d+1}$. Then,

$$ f_d(\Sigma)\leq \frac{d+2}{d} (f_{d-1}- \beta_{d-1}-1).
$$

\end{cor}
\begin{proof}
This is the $k=1$  case of Theorem~\ref{thm1b}, using that
% and the observation that 
$\mathrm{g}(\Sigma) \ge d+2$.
% and $\chi_{d-2}\ge 0$.
%$\mathrm{girth}(\Sigma)$ is at least $d+2$.
\end{proof}

Next, we focus on balanced simplicial complexes. Recall that a $d$-dimensional simplicial complex is said to be \emph{balanced} if its underlying graph ($1$-skeleton) is $(d+1)$-colorable in the graph theoretic sense. 

\begin{thm}\label{flag1}

Let $\Sigma$ be a balanced $d$-dimensional simplicial complex that admits a PL embedding into $\R^{d+1}$. Then the following hold true:
\begin{enumerate}[{\rm (a)}]
\item\label{item1f} $2^{d}(\beta_d(\Sigma;\Z_2)+1)\leq f_d(\Sigma)$;
\item\label{item2f}$f_d(\Sigma)\leq \frac{2^{d}}{2^{d}-1}
(\chi_{d-1} -1)$.
\end{enumerate}
\end{thm}
\begin{proof}
It suffices to show that the girth of a balanced $d$-dimensional simplicial complex is at least $2^{d+1}$. The crucial point is that a balanced $d$-dimensional complex with non-zero top dimensional homology has at least $2^{d+1}$ faces of dimension $d$. To see this one can observe that such a complex must contain a balanced $d$-dimensional pseudomanifold without boundary; the pure complex whose facets are support of a $d$-cycle. The claim then can be proved easily for pseudomanifolds, say by induction on the dimension. We leave it to the reader to fill in the details. 
\end{proof}

Our method is applicable also to complexes that admit a codimension zero embedding. For this, we first need to prove an 
auxiliary result. 

\begin{lem}\label{BetaEf}
Let $\Sigma$ be a $d$-dimensional simplicial complex and 
let $\Sigma^{(-1)}$ denote its codimension one skeleton. Then one has
\[f_d(\Sigma) = \beta_d(\Sigma;\Z_2)-\beta_{d-1} (\Sigma;\Z_2)
+\beta_{d-1}(\Sigma^{(-1)};\Z_2).\]
\end{lem}
\begin{proof}
%%%%%%%%%%%%%%%%%

We have that $f_i(\Sigma^{(-1)})=f_i(\Sigma)$ for all $i\leq d-1$, 
and $\beta_i(\Sigma^{(-1)})=\beta_i(\Sigma)$
for all $i\leq d-2$. 
Hence, by the Euler-Poincar\'e formula

\[
 (-1)^{d}f_{d}(\Sigma) =\chi(\Sigma)-\chi(\Sigma^{(-1)}) =
(-1)^{d} \left(\beta_{d}(\Sigma)-\beta_{d-1}(\Sigma)
+\beta_{d-1}(\Sigma^{(-1)} \right)
 %+f_{d}(\Sigma)
%=\beta_{d-1}(\Sigma)+f_{d}(\Sigma)-1.\]
\]

\end{proof}

%%%%%%%%%%%%%%%%%
%Using the fact that the $f$- and $\beta$-numbers are invariant under 
%exterior algebraic shifting,
%we can assume that $\Sigma$ is a shifted complex on the vertex set $%%%\{
%and properties of shifted complexes and the exterior algebraic shifting %operation. Let $A_d$ be the set of $d$-dimensional faces of $\Sigma$ that %contain the vertex $1$, and let $B_d$ be those that do not contain it. On the %one hand the cardinality of $B_d$ is equal to $\beta_d(\Sigma;\Z_2)$. On the %other hand $\beta_{d-1}(\Sigma^{(-1)};\Z_2)$ is equal to the cardinality of the %set $B_{d-1}$; the set of all $(d-1)$-dimensional faces of $\Sigma$ that do not %contain the vertex $1$. However, the map $A_d\rightarrow B_{d-1}$ that %sends $\sigma$ to $\sigma\setminus \{1\}$ is an injection, since $\Sigma$ is a %simplicial complex. This concludes the argument. 
%\end{proof}

\begin{cor}\label{dtod}
Let $\Sigma$ be a $d$-dimensional simplicial complex that admits a PL embedding into $\R^{d}$. Then $f_d(\Sigma)\leq \frac{2}{d+1} f_{d-1}
(\Sigma) -1$.
\end{cor}
\begin{proof}
It can easily be shown, say by using Alexander duality, that the top dimensional homology of $\Sigma$ must be zero. Thus, it follows from Lemma~\ref{BetaEf} that $\beta_{d-1}(\Sigma^{(-1)};\Z_2)\geq f_d(\Sigma)$. Now, applying Theorem~\ref{thm1a} to $\Sigma^{(-1)}$ we get
\[(d+1)(f_d(\Sigma)+1)\leq (d+1)(\beta_{d-1}(\Sigma^{(-1)};\Z_2)+1)\leq 2f_{d-1}(\Sigma^{(-1)})=2f_{d-1}(\Sigma).\]
\end{proof}

\section{Corollaries in Low Dimensions}
In this section we summarise direct consequences of the main results 
for embeddings into
dimensions $2$, $3$ and $4$. Throughout, the number of vertices of 
a simplicial complex is denoted by $n$ (rather than $f_0$).

\begin{prop}\label{2to2}
Let $\Sigma$ be a $2$-dimensional complex that PL embeds into $\R^2$. Then $f_2(\Sigma)\leq \frac{2}{3}f_1(\Sigma)-1$. In particular, $f_2(\Sigma)\leq 2n-5$.
\end{prop}
\begin{proof}
The first inequality follows from Corollary~\ref{dtod}. The second inequality follows from the fact that the underlying graph of $\Sigma$ is planar. 
\end{proof}

\begin{prop}\label{2to3}
Let $\Sigma$ be a connected $2$-dimensional complex that PL embeds into $\R^3$. Then $f_2(\Sigma)\leq 2(f_1(\Sigma)-\beta_1(\Sigma)-n)$. 
\end{prop}
\begin{proof}
This follows easily from Theorem~\ref{thm1b}.
\end{proof}
\begin{cor}[Dey-Edelsbrunner~\cite{DE94}]\label{de}
Let $\Sigma$ be a $2$-dimensional complex that PL embeds into $\R^3$. Then $f_2(\Sigma)\leq n(n-3)$. 
\end{cor}
\begin{proof}
Without loss of generality, we may assume that $\Sigma$ is connected. 
The inequality is an immediate consequence of Proposition~\ref{2to3} 
and the trivial 
%${n\choose 2}$ 
$\binom{n}{2}$ upper bound for $f_1(\Sigma)$. 
\end{proof}

\begin{cor}
Let $\Sigma$ be a $3$-dimensional complex that PL embeds into $\R^3$. Then $f_3(\Sigma)\leq n(n-3)/2-1$. 
\end{cor}
\begin{proof}
This follows from Corollary~\ref{dtod} and Corollary~\ref{de}.
\end{proof}

\begin{prop}
Let $\Sigma$ be a connected balanced $2$-dimensional complex that embeds into $\R^3$. Then $f_2(\Sigma)\leq \frac{4}{9}(n^2-3n)$.
\end{prop}
\begin{proof}
It follows from Theorem~\ref{flag1} that $f_2(\Sigma)\leq \frac{4}{3}(f_1(\Sigma)-n)$. Now, since the underlying graph of $\Sigma$ is $3$-colorable, one has $f_1(\Sigma)\leq 3(\frac{n}{3})^2$. The conclusion now follows easily. 
\end{proof}

For embeddings into dimension $4$ much less is known.
 It was conjectured by Kalai and Sarkaria (see Kalai's blog~\cite{CombinAndMore}, for instance) that if a $2$-dimensional complex is embedded into $\R^4$, then it has at most $2n(n-1)$ triangles. This conjecture is wide open. Currently, the best known bound~\cite{parsa} is $C\cdot n^{8/3}$, where $C$ is a constant. Here is what
 our method yields in the case of embeddings into 
 dimension four.
 
 \begin{prop}\label{3to4}
Let $\Sigma$ be a connected $3$-dimensional complex that PL embeds into $\R^4$. 
Then $f_3(\Sigma)\leq \frac{5}{3}\left( f_2(\Sigma)-f_1(\Sigma)
-\beta_2(\Sigma) +\beta_1(\Sigma)+n-2 \right)$. 
\end{prop}
\begin{proof}
This follows from Theorem~\ref{thm1b}.
\end{proof}
\begin{cor}
Let $\Sigma$ be a connected $3$-dimensional complex that PL embeds into $\R^4$. 
%Then $$ f_3(\Sigma)\leq \frac{5}{3}(f_2(\Sigma) + n -2),
Then, $$f_3(\Sigma)\leq \frac{5}{3}(f_2(\Sigma) + \beta_{1}(\Sigma) -1) 
\, \mbox{ and } \;
f_3(\Sigma)\leq \frac{5}{3} \left({n\choose3}+ n -2) \right).$$
If $\Sigma$ is simply connected, then
$f_3(\Sigma)\leq \frac{5}{3}(f_2(\Sigma)  -1) $.
\end{cor}
\begin{proof}
%Without loss of generality, we may assume that $\Sigma$ is connected. 
The inequalities are immediate consequences of Proposition~\ref{3to4} 
and the trivial ${n\choose 3}$ upper bound for $f_2(\Sigma)$. 
\end{proof}

% We close this section by presenting two conjectures. 

%\begin{conj}\label{conj1}
%Let $\Sigma$ be a $3$-dimensional complex that PL embeds into $\R^4$. Then %$f_3(\Sigma)< 5n(n+2)$.
%\end{conj}

%\begin{conj}\label{conj2}
%Let $\Sigma$ be a $4$-dimensional complex that PL embeds into $\R^4$. Then %$f_4(\Sigma)< 2n(n+2)$.
%\end{conj}

\section{Estimates} 

In the following we give an upper bound for the number of top dimensional faces of a $d$-dimensional simplicial complex embedded into $\R^{d+1}$ in terms of the number of its vertices. Let us begin by observing that for a $d$-complex $\Sigma$ on $n$ vertices one has $f_{d-1}(\Sigma)\leq {n \choose d}$. Hence, it follows from Theorem~\ref{thm1b} that $f_d(\Sigma)<
(1+\frac{2}{d}){n \choose d}$. Therefore, we can easily obtain the upper bound $f_d(\Sigma)=\mathcal{O}(n^d)$ due to Dey and Pach~\cite[Theorem 3.1]{DeyPach}, where $\mathcal{O}$ is the big O notation. 

Below we present a slightly better upper bound by using our Theorem~\ref{thm1b} and a combination of an idea due to Gundert~\cite{Anna} and Sperner's Lemma~\cite[Lemma 4.5]{GK}. Recall that Sperner's Lemma asserts that for a simplicial complex $\Sigma$ on $n$ vertices the quantity $f_{i}(\Sigma)/f_{i-1}(\Sigma)$ is at most ${n\choose i+1}/{n\choose i}$. Notice that Sperner's Lemma can easily 
be strengthened to 
\[
f_i(\Sigma)/f_j(\Sigma)\leq {n\choose i+1}/{n \choose j+1}=\mathcal{O}(n^{i-j}),
\]
for all $i>j$.

\begin{thm}\label{thm11}
Let $\Sigma$ be a $d$-dimensional simplicial complex that admits a PL embedding into $\R^{d+1}$. Then $f_d(\Sigma)=\mathcal{O}(n^{d-\epsilon})$, where $\epsilon=3^{-\lceil\frac{d+1}{2} \rceil}$. 
\end{thm}

\begin{proof}
Let us, to simplify notation, put $\ell=\lceil\frac{d+1}{2} \rceil$. Let $\Delta$ be the $\ell$-dimensional skeleton of $\Sigma$. Since $\Delta$ is embeddable into $\R^d$, it follows from~\cite[Proposition 3.3.5]{Anna} that 
\[
f_\ell(\Sigma)=f_\ell(\Delta)=\mathcal{O}(n^{\ell+1-3^{-\ell}}).
\]
Now, it follows from Sperner's Lemma that $f_{d-1}(\Sigma)= \mathcal{O}(n^{d-\ell-1})f_\ell(\Sigma)$. Therefore, one obtains that $f_{d-1}(\Sigma)= \mathcal{O}(n^{d-3^{-\ell}})$. Finally, the conclusion follows from Theorem~\ref{thm1b}. 
\end{proof}

We remark that the upper bound provided in Theorem~\ref{thm11} is probably far from the true upper bound. Actually, it was shown by Dey and Pach~\cite{DeyPach} that if a $k$-dimensional complex $\Sigma$ embeds into $\R^k$ then $f_k(\Sigma)=\mathcal{O}(n^{\lceil \frac{k}{2} \rceil})$. Indeed, for $k\geq 4$ it is an open problem to show that if a simplicial complex embeds into $\R^k$, then the total number of its faces is bounded above by $\mathcal{O}(n^{\lceil \frac{k}{2} \rceil})$.

\section{An upper bound by Gr\"{u}nbaum}

{\rm

%After the presentation of this paper at the Ji{\v{r}}{\'{\i}} Matou{\v{s}}ek memorial
%conference in Prague, in July 2016, Uli Wagner brought to our attention the 

In the 1970 paper
~\cite{Gru} Branko Gr\"{u}nbaum
%  ~\cite[5(i)]{Gru}  
%Gr\"{u}nbaum
shows that if a $d$-dimensional complex $\Sigma$ embeds into $\R^{d+1}$, then $f_d(\Sigma)\leq \frac{6}{d+1}f_{d-1}(\Sigma)$. He also proves
slightly sharper versions of this result for pure complexes, see Proposition~\ref{fG} below. 
%In the rest of this section we compare our bound to the one given by Gr\"{u}nbaum\citep{•}}. 

How do the different bounds compare?
Due to their different structure it is hard to
make a general comparison. 
In view of having leading constant  $\frac{6}{d+1}$, 
%Gr\"{u}nbaum rather than  $\frac{d+2}{d}$,  
it is clear that Gr\"unbaum's upper 
bound is better than ours in several cases, particularly
when one has only some partial $f$-vector information.
%involving only the last two entries of the $f$-vector. 
%of that complex. 
However, our bound is tighter in other cases, especially if much
structural  information, expressed in terms
of  $f$- and $\beta$-vectors, is available. In this section, we present one 
such case.

Let us begin with the following 
result, which extends the validity 
of Gr\"{u}nbaum's
inequality~\cite[5(iii)]{Gru} to embeddability into manifolds.

%The proof is standard and straightforward but we include it for the sake of completeness. 

\begin{prop}\label{fG}
Let $\Sigma$ be a pure $d$-dimensional simplicial complex that is PL embeddable 
into a $(d+1)$-dimensional PL manifold. Then 
\begin{equation}\label{Geq}
f_d(\Sigma) \le \frac{6}{d+1} \  f_{d-1}(\Sigma) - \frac{10}{d(d+1)} \  f_{d-2} (\Sigma).
\end{equation} 
\end{prop}

\begin{proof}
We know that if $\Sigma$ is a planar graph which contains at least one edge, then\footnote{Note that "$-5$" is needed here, instead of "$-6$",
in order to include the case when $f_0=2$ for the inductive argument.} 
 $f_1(\Sigma)\leq 3f_0(\Sigma)-5$. This verifies the first step $d=1$ of an
 inductive argument.

Now assume that the statement is valid for every $1\leq k<d$ and $\Sigma$ is a pure $d$-dimensional simplicial complex that is PL embeddable into a $(d+1)$-dimensional
PL manifold. Let $V$ denote the vertex set of $\Sigma$ and for $v\in V$, let $L_v$ be the link of $v$ in $\Sigma$. Since $\Sigma$ is embeddable into a $(d+1)$-manifold, $L_v$ must be embeddable into a $d$-sphere and we have 
\[d!f_{d-1}(L_v)\leq 6(d-1)!f_{d-2}(L_v)-10(d-2)!f_{d-3}(L_v).\]
Summing over all vertices $v\in V$ and using the equation $\sum_vf_i(L_v)=(i+2)f_{i+1}(\Sigma)$ yields the desired conclusion. 
\end{proof}

Say we are interested in the question whether the $d$-skeleton
of a $(d+1)$-manifold is embeddable into $\R^{d+1}$. 
If the manifold in question has non-vanishing homology in dimension $d$
(or equivalently in dimension one) our
inequalities turn out to be sharp enough to
provide a negative answer, while 
Gr\"{u}nbaum's inequality \eqref{Geq} is not. 
%making it possible to prove  the following result 
% Then the $f$- and $\beta$-vectors of $\Sigma$ do not satisfy the inequalities \eqref{eq2} and \eqref{eq3}, and, in particular, 

\begin{prop}
Let $\Sigma$ be the $d$-skeleton
of a triangulated $(d+1)$-dimensional PL manifold with non-zero $d$-dimensional Betti number.
% Then the $f$- and $\beta$-vectors of $\Sigma$ do not satisfy the inequalities \eqref{eq2} and \eqref{eq3}, and, in particular, 
Then
$\Sigma$ is not PL-embeddable into $\R^{d+1}$. 
%However, the $f$-vector of $\Sigma$ does satisfy Gr\"{u}nbaum's 
%inequality \eqref{Geq}.
%stated in Proposition~\ref{fG}.
\end{prop}
\begin{proof}
Let $M$ denote the $(d+1)$-dimensional manifold in question.
%%%%%%%%%%%%%
We know from Lemma \ref{BetaEf} that
\[f_{d+1}(M) = \beta_{d+1}(M;\Z_2)-\beta_{d} (M;\Z_2) +\beta_{d}(\Sigma;\Z_2).\]
Since $M$ is a manifold, one has $(d+2)f_{d+1}(M)=2f_{d}(M)$
and $\beta_{d+1}(M)=1$. This, together with 
the assumption $\beta_{d}(M)\ge 1$
 imply that 
\[
(d+2)\left(\beta_{d}(\Sigma)+1 \right) =
(d+2)\left(\beta_{d+1}(M)+\beta_d(\Sigma) \right)>
(d+2) f_{d+1}(M) =2f_{d}(M), \]
which  violates  the inequality \eqref{eq2} of Theorem 2.
Therefore, $\Sigma$ is not PL embeddable into $\R^{d+1}$.
Also the inequality \eqref{eq3}  is violated.

Observe, however, that 
Gr\"{u}nbaum's inequality  \eqref{Geq} is satisfied by $f(\Sigma)$.  
This follows from 
Proposition \ref{fG}, since $\Sigma$ is PL  embeddable into a 
$(d+1)$-dimensional manifold, namely $M$.
\end{proof}

\begin{ex}{\rm
As a concrete example of this type, one may take $T$ to be a triangulation 
of the $3$-torus with $f(T)=(15,105,180,90)$. Such a triangulation exists and happens to be the smallest (w.r.t. the $f$-vector) known triangulation of the $3$-torus $S^1\times S^1\times S^1$. See~\cite[Table 7]{Lutz} for instance. Let $\Sigma$ be the $2$-skeleton of $T$. Then one has $f(\Sigma)=(15,105,180)$ and $\beta(\Sigma)=(1,3,92)$. 
}
\end{ex}

\subsection*{Acknowledgement} We thank the referees for very useful remarks. 
Also, we thank Uli Wagner for valuable comments and for
bringing Gr\"{u}nbaum's paper ~\cite{Gru}  to our attention. 

This research was supported by the grant
VR-2015-05308 from Vetenskapsr{\aa}det. Part of the research of the second author has been made possible by the grant KAW-stipendiet 2015.0360 from the Knut and Alice Wallenberg's Fondation.

%%%%%%%%%%%%%%%%%%%%%%%%%%%%%%%%%%%%%%%%%%%%%%%%%

\def\cprime{$'$}
\providecommand{\bysame}{\leavevmode\hbox to3em{\hrulefill}\thinspace}
\providecommand{\MR}{\relax\ifhmode\unskip\space\fi MR }
\providecommand{\MRhref}[2]{%
  \href{http://www.ams.org/mathscinet-getitem?mr=#1}{#2}
}

\end{document}